\documentclass[12pt,a4paper,draft]{amsart}
\usepackage[english]{babel}
\usepackage{a4wide,amssymb,xypic,ifthen,hyperref}
\usepackage[mathcal]{euscript}

\newcommand{\qee}{ \hfill\hspace{2pt}$\triangle$}
\newcommand{\marginnote}[1]{\ifthenelse{\isodd{\thepage}}{\normalmarginpar}
{\reversemarginpar}\marginpar{\fbox{\parbox{15mm}{\sloppy\footnotesize #1}}}}

\newtheorem{thm}{Theorem}[section]
\newtheorem{corol}[thm]{Corollary}
\newtheorem{lemma}[thm]{Lemma}
\newtheorem{prop}[thm]{Proposition}
\newtheorem{defin}[thm]{Definition}

\theoremstyle{remark}
\newtheorem{rema}[thm]{Remark}
 \newenvironment{remark}{\begin{rema}}{\qee\end{rema}}
\newtheorem{exe}[thm]{Example}

\newcommand{\Oc}{\mathcal O}

\newcommand{\PP}{{\mathbb P}}

\newcommand{\R}{\mathbb R}
\newcommand{\C}{\mathbb C}
\newcommand{\Q}{\mathbb Q}
\newcommand{\Z}{\mathbb Z}

\def\ker{\operatorname{ker}}

\def\rk{\operatorname{rk}}
\def\Ad{\operatorname{Ad}}

\def\End{\operatorname{End}}
\def\dim{\operatorname{dim}}
\def\rad{\operatorname{rad}}

\def\Aut{\operatorname{Aut}}
\newcommand\grass{\mbox{Gr}}
\newcommand\hgrass{{\mathfrak{Gr}}}
\newcommand{\cO}{{\mathcal O}}

\newcommand{\fE}{{\mathfrak E}}

\newcommand{\fR}{{\mathfrak R}}
\newcommand{\fQ}{{\mathfrak Q}}

\newcommand{\fV}{{\mathfrak V}}

\newcommand{\g}{{\mathfrak g}} 

\newcommand{\cat}[1]{\boldsymbol{\operatorname{#1}}}

\renewcommand{\bysame}{\leavevmode\vrule height 2pt depth-1.6pt width 23pt,\ }

\begin{document}
\parbox{1cm}\hfill\vspace{-30pt}
\title[Semistable and nef principal Higgs bundles]{Semistable  and numerically 
effective \\[8pt] principal (Higgs) bundles}
\bigskip
\date{\today}
\thanks{This research was partly supported by the Spanish {\sc
mec} through the research project MTM2009-07289,   by 
``Grupo de Excelencia de Castilla y Le\'on" GR46, by
Istituto Nazionale per l'Alta Matematica and   by the European project {\sc misgam}.  Both authors are members of the research
group {\sc vbac} (Vector Bundles on Algebraic Curves).}

 \maketitle \thispagestyle{empty} \vspace{-3mm}
\begin{center}{\sc Ugo Bruzzo} \\
Scuola Internazionale Superiore di Studi Avanzati,\\ Via Beirut 2-4, 34013
Trieste, Italia; \\ Istituto Nazionale di Fisica Nucleare, Sezione di Trieste \\ E-mail: {\tt bruzzo@sissa.it} 
\\[6pt]
{\sc Beatriz Gra\~na Otero} \\
Departamento de Matem\'aticas  and Instituto de F\'\i sica \\ Fundamental y Matem\'aticas, Universidad de Salamanca,
\\ Plaza de la Merced 1-4, 37008 Salamanca, Espa\~na\\ E-mail: {\tt beagra@usal.es}
\end{center}

\vfill

\begin{abstract} We study Miyaoka-type semistability criteria for principal Higgs $G$-bundles $\fE$ on complex projective manifolds of any dimension.
We prove that $\fE$ has the property of being semistable after pullback to any projective curve  if and only if certain line bundles, obtained from some characters
of the parabolic subgroups of $G$, are numerically effective. One also proves that
these conditions are met for semistable principal Higgs bundles whose adjoint bundle has vanishing second Chern class. 

In a second part of the paper,  we introduce   notions of numerical effectiveness and numerical flatness for principal (Higgs) bundles, discussing their main properties.
For (non-Higgs) principal bundles, we show that  
a numerically flat principal   bundle admits a reduction to a Levi factor which has a flat  Hermitian-Yang-Mills connection, and, as a consequence, that
the cohomology ring  of a numerically flat principal   bundle with coefficients in $\R$ is trivial. 
To our knowledge this notion of numerical effectiveness is new even in the case of (non-Higgs) principal bundles. \end{abstract}

\newpage \tableofcontents

\section{Introduction}
In 1987 Miyaoka gave a criterion for the semistability of a vector bundle $V$ on a projective curve in terms
of the numerical effectiveness of a suitable divisorial class (the relative anticanonical divisor of the
projectivization $\mathbb PV$ of $V$) \cite{Mi}. Recently several generalizations of this criterion have been formulated \cite{BH,BB1,BSch},
dealing with principal bundles, higher dimensional varieties, and considering also the case of bundles
on compact K\"ahler manifolds.

 In this paper we prove a Miyaoka-type criterion for principal Higgs bundles on complex
projective manifolds. Let us give a rough anticipation of this result. Given a principal Higgs $G$-bundle $\fE=(E,\phi)$
on a complex projective manifold $X$, with Higgs field $\phi$,
 and a parabolic subgroup $P$ of $G$, we introduce a subscheme
${\fR}_P(E,\phi)$ of the total space of the bundle $E/P\to X$ whose sections parametrize reductions of the structure group $G$ to $P$
that are compatible with the Higgs field $\phi$. Then in Theorem \ref{Miyahigher} we prove the equivalence
of the following conditions: for every reduction of $G$ to a parabolic subgroup $P$ which is compatible with the Higgs field, and every dominant character of $P$, a certain associated line bundle on ${\fR}_P(E,\phi)$ is numerically effective; the pullback $f^\ast\fE$ is semistable for any morphism $f\colon C\to X$, where $C$ is any smooth projective curve. One also shows that both conditions are met when $\fE$
is a semistable principal Higgs bundle such that $c_2(\Ad(E))=0$.

In a  second part of this paper, we define notions of numerical
effectiveness and numerical flatness which are appropriate for
principal Higgs bundles. It is known \cite{DPS94} that a
numerically flat   vector bundle admits a filtration whose
quotients are  stable Hermitian flat   vector bundles. 
In section \ref{flatcon} we prove that to a numerically flat  principal (non-Higgs) bundle
 one can associate a principal   bundle,
whose structure group is the Levi factor of a parabolic
subgroup of $G$, which is polystable,
and admits a flat ``Hermitian'' connection. This implies
that the   characteristic
ring (with coefficients in $\R$)  of the principal bundle vanishes.

Section \ref{tannaka} develops some Tannakian considerations; basically we show
the equivalence of proving our   theorem \ref{Miyahigher} for principal Higgs bundles
or for  Higgs vector bundles.

In an Appendix (Section \ref{appendix}) we offer a resume of our previous work on Higgs vector bundles
\cite{BG1,BG2,BH}, on which some parts of the present paper rely quite heavily.

As a principal Higgs bundle with zero Higgs field is exactly a
principal bundle, all results we prove in this paper hold true for
principal bundles. In this way we mostly recover
well-known results or some of the results in \cite{BB1,BB2} with their
proofs, at other times we provide simpler demostrations, while at
times the results are altogether new. The notion of numerical effectiveness
we introduce is, on the other hand, new also for the case of principal bundles.

{\bf Acknowledgements.} This paper was mostly written during a
visit of both authors at the University of Pennsylvania. We thank
Penn for hospitality and support, and the staff and the
scientists at the Department of Mathematics for providing an
enjoyable and productive atmosphere. We thank M.S.~Narasimhan, Tony Pantev and Carlos Simpson for
valuable suggestions. We also thank the  Department of Mathematics of Universit\'e d'Angers and the Department of Physics and Astronomy
of Rutgers University for hospitality while this paper was finalized.

\bigskip\section{Semistable principal bundles}\label{Sec:2}
In this short section we recall some basics about principal
bundles, notably the definition of (semi)stable principal bundle (basic references about this topic are \cite{Rama,BaShe}). Let $X$ be a
smooth complex projective variety, $G$ a complex reductive
algebraic group, and $\pi\colon E\to X$ a principal $G$-bundle on
$X$. If $\rho\colon G \to \operatorname{Aut}(Y)$ is a
representation of $G$ as automorphisms of a variety $Y$, we may
construct the associated bundle $E(\rho)=E\times_\rho Y$, the
quotient of $E\times Y$ under the   action of $G$ given by
$(u,y)\mapsto (ug,\rho(g^{-1})y)$ for $g\in G$. If $Y=\g$ is the
Lie algebra of $G$, and $\rho$ is the adjoint action of $G$ on
$\g$, one gets the adjoint bundle of $E$, denoted by  $\Ad(E)$.
Another important example is obtained when $\rho$ is given by a
group homomorphism $\lambda\colon G\to G'$; in this case the
associated bundle $E'=E\times_\lambda G'$ is a principal
$G'$-bundle. We say that the structure group $G$ of $E$ has been
extended to $G'$.

If $E$ is a principal $G$-bundle on $X$, and $F$ a principal
$G'$-bundle on $X$, a morphism $E\to F$ is a pair $(f,f')$, where
$f'\colon G\to G'$ is a group homomorphism, and $f\colon E\to F$
is a morphism of bundles on $X$ which is $f'$-equivariant, i.e.,
$f(ug)=f(u)f'(g)$. Note that this induces a vector bundle morphism
$\tilde f\colon \Ad(E)\to\Ad(F)$ given by $\tilde
f(u,\alpha)=(f(u),f'_\ast(\alpha))$, where $f'_\ast\colon\g\to\g'$
is the morphism induced on the Lie algebras.
As an example, consider a principal $G$-bundle $E$, a group homomorphism $\lambda\colon G\to G'$,
and the extended bundle $E'$. There is a natural morphism $(f,\lambda)\colon E\ \to E'$, where
$f=\operatorname{id}\times\lambda$ if we identify $E$ with $E\times_GG$.

If $K$ is a closed subgroup of $G$, a \emph{reduction} of the structure
group $G$ of $E$ to $K$ is a principal $K$-bundle $F$ over $X$ together
with an injective $K$-equivariant bundle morphism $F \to E$.
Let $E(G/K)$ denote
the bundle over $X$ with standard fibre $G/K$ associated to $E$ via
the natural action of $G$ on the homogeneous space $G/K$.
There is an isomorphism   $E(G/K)\simeq E/K$ of bundles over $X$. Moreover,
the reductions of the structure group of $E$ to $K$ are in a one-to-one correspondence
with sections $\sigma\colon X \to E(G/K)\simeq E/K$.

 We first recall the definition of
semistable principal bundle when the base variety $X$ is a curve.
Let $T_{E/K,X}$ be the vertical tangent bundle to the bundle
$\pi_K\colon E/K \to X$.

\begin{defin} Let $E$ be a principal $G$-bundle on  a smooth connected projective curve $X$. We say that
$E$ is stable (semistable) if for every proper parabolic subgroup $P\subset G$, and every reduction $\sigma\colon X \to E/P$,
the pullback $\sigma^\ast (T_{E/P,X})$ has positive (nonnegative) degree.
\end{defin}

When $X$ is a higher dimensional variety, the definition must be somewhat
refined; the introduction of an open dense subset whose complement
has codimension at least two should be compared with the definition
of (semi)stable vector bundle, which involves non-locally free subsheaves
(which are subbundles exactly on open subsets of this kind).

\begin{defin} Let $X$ be a polarized smooth projective variety. A principal  $G$-bundle $E$  on $X$ is stable (semistable)
 if and only   if  for any proper parabolic subgroup $P\subset G$, any open dense subset $U\subset X$ such that $\operatorname{codim}(X-U)\ge 2$, and any  reduction $\sigma\colon U \to (E/P)_{\vert U}$ of $G$ to $P$ on $U$, one has $\deg \sigma^\ast (T_{E/P,X}) > 0$ ($\deg \sigma^\ast (T_{E/P,X}) \ge 0$).
\end{defin}

Here it is important that the smoothness of $X$ guarantees that a
line bundle defined on an open dense subset of $X$, whose complement has codimension 2 at least,  extends uniquely to the whole of $X$, so that we may consistently consider
its degree. This is discussed in detail in \cite{RamaSub}, see also \cite{koba}, Chapter V.

\bigskip\section{Principal Higgs bundles}\label{phb}

We switch now to principal Higgs bundles.
Let $X$ be a smooth complex projective variety, and $G$ a reductive
complex  algebraic group. If $E$ is a principal $G$-bundle on $X$,
$\Ad(E)$ is its adjoint bundle, and  $\phi$, $\psi$ are  global sections
 of $\Ad(E)\otimes\Omega^1_X$, we can define a   section
 $[\phi,\psi]$ of $\Ad(E)\otimes\Omega^2_X$ by combining the bracket
 $[\,,\,]\colon \Ad(E)\otimes\Ad(E)\to\Ad(E)$ with the natural morphism
 $\Omega^1_X\otimes\Omega^1_X\to\Omega^2_X$.

\begin{defin} A principal Higgs $G$-bundle $\fE$ is a pair $(E,\phi)$, where
$E$ is a principal $G$-bundle, and $\phi$  is a global section
 of $\Ad(E)\otimes\Omega^1_X$ such that $[\phi,\phi]=0$.
\end{defin}

When $G$ is the general linear group, under the identification $\Ad(E)\simeq
\End(V)$, where $V$ is the vector bundle corresponding to $E$, this agrees with the usual definition  of Higgs vector bundle.

\begin{defin} \label{trivial} A principal Higgs $G$-bundle $\fE=(E,\phi)$ is \emph{trivial}
if $E$ is trivial, and $\phi=0$.
\end{defin}

A morphism between two principal Higgs bundles $\fE=(E,\phi)$ and $\fE'=(E',\phi')$
is a principal bundle morphism $f\colon E\to E'$ such that $(f_\ast\times \operatorname{id})(\phi)=\phi'$, where
$f_\ast\colon \Ad(E)\to\Ad(E')$ is the induced morphism between the adjoint bundles.

We introduce the notion of extension of the structure group
for a principal Higgs $G$-bundle $\fE=(E,\phi)$. Given a group
homomorphism $\lambda\colon G\to G'$, we consider the extended
principal bundle $E'$. The group $G$ acts on the Lie algebra $\g'$
of $G'$ via the homomorphism $\lambda$ (and the adjoint action of
$G'$), and the $\g'$-bundle associated to $E$ via the adjoint
action of $G'$ is isomorphic to $\Ad(E')$. In this way the Higgs
field of $\fE$ induces a Higgs field for $\fE'$. More generally,
if $\rho\colon G\to\operatorname{Aut}(V)$ is a linear
representation of $G$, the Higgs field of $\fE$ induces a Higgs
field for the associated vector bundle $E\times_\rho V$.

If $\fE$ is a principal Higgs $G$-bundle, we denote by
$\Ad(\fE)$ the Higgs vector bundle given by the adjoint bundle
$\Ad(E)$ equipped with the induced Higgs morphism.

Let $K$ be a closed subgroup of $G$, and $\sigma\colon X \to
E(G/K)\simeq E/K$ a reduction of the structure group of $E$ to
$K$. So one has a principal $K$-bundle $F_\sigma$ on $X$ and a
principal bundle morphism $i_\sigma\colon F_\sigma\to E$ inducing
an injective morphism of bundles $\Ad(F_\sigma) \to \Ad(E)$. Let
$\Pi_\sigma\colon  \Ad(E)\otimes
\Omega^1_X\to(\Ad(E)/\Ad(F_\sigma))\otimes \Omega^1_X $ be the
induced projection.

\begin{defin} A section $\sigma\colon X\to E/K $ is a {\em Higgs reduction} of $(E,\phi)$
if $\phi\in\ker \Pi_\sigma$.
\end{defin}

When this happens, the reduced bundle $F_\sigma$ is equipped with a Higgs field $\phi_\sigma$ compatible with $\phi$ (i.e., $(F_\sigma,\phi_\sigma)\to (E,\phi)$ is a morphism of principal Higgs bundles).

\begin{remark} Let us again consider the case when $G$ is the general linear group $Gl(n,\C)$, and let us assume that
$K$ is a maximal parabolic subgroup, so that $G/K$ is the Grassmann variety $\operatorname{Gr}_k(\C^n)$ of $k$-dimensional quotients
of $\C^n$ for some $k$. If $V$ is the vector bundle corresponding to $E$, a reduction $\sigma$ of $G$ to $K$ corresponds to a rank $n-k$
subbundle $W$ of $V$, and the fact that $\sigma$ is a Higgs reduction means that $W$ is $\phi$-invariant. \end{remark}

The choice of $\phi$ singles out a subscheme
of the variety $E/K$, which describes the Higgs reductions of the
pair $(E,\phi)$. Let $E_K$ denote the principal $K$-bundle
$E\to E/K$. Since the vertical tangent bundle $T_{E/K,X}$ is the
bundle associated to $E_K$ via the adjoint action of $K$ on
the quotient $\mathfrak g/\mathfrak k$, and $\pi_K^\ast \Ad(E)$ is
the bundle associated to $E_K$ via the adjoint action of $K$
on $\mathfrak g$,   there is a natural morphism $\eta\colon
\pi_K^\ast \Ad(E) \to T_{E/K,X}$. Then $\phi$ determines a section
$\eta(\phi) := (\eta \otimes \text{id})(\pi_{K}^{*}\phi)$ of
$T_{E/K,X}\otimes \Omega^1_{E/K}$.

\begin{defin} \label{Higgsred} The {\em scheme of  Higgs reductions of $\fE=(E,\phi)$ to $K$}
is the closed subscheme  ${\fR}_K(\fE)$ of $E/K$ given by the zero locus of $\eta(\phi)$.
\end{defin}

\begin{remark}\label{inducedHiggs} The Higgs field of $\fE$
induces a Higgs field on the restriction of $E_K$ to ${\fR}_K(\fE)$;
we denote by $\fE_K$ the resulting principal Higgs $K$-bundle.
\end{remark}

The construction of the scheme of  Higgs reductions is compatible with base change.
Let us recall that given a principal Higgs $G$-bundle $\fE=(E,\phi)$ over $X$,
and a morphism $f\colon Y \to X$, the pullback Higgs bundle $f^\ast\fE$ is
the pullback principal bundle $f^\ast E$ equipped with a Higgs field 
obtained by combining the pullback morphism 
$$\Ad(f^\ast E) \simeq  f^\ast \Ad(E) \to  \Ad(f^\ast E)\otimes f^\ast\Omega^1_X$$
with the natural morphism $f^\ast\Omega^1_X\to\Omega^1_Y$.
The above mentioned compatibility means that, if $f$ is a
morphism of smooth complex projective
varieties,  then
${\fR}_K(f^\ast(\fE))\simeq Y\times_X {\fR}_K(\fE)$. By
construction, $\sigma\colon X \to E(G/K)\simeq E/K$ is a Higgs
reduction if and only if it takes values in the subscheme
${\fR}_K(\fE)\subset E/K$. Moreover the scheme of Higgs reductions
is compatible with morphisms of principal Higgs bundles.  This
means that if $\fE=(E,\phi)$ is a principal Higgs $G$-bundle,
$\fE'=(E',\phi')$ a principal Higgs $G'$-bundle, $\psi\colon G\to
G'$ is a group homomorphism, and $f\colon\fE\to\fE'$ is a
$\psi$-equivariant morphism of principal Higgs bundles, then for
every closed subgroup $K\subset G$  the induced morphism $E/K\to
E'/K'$, where $K'=\psi(K)$, maps ${\fR}_K(\fE)$ into
${\fR}_{K'}(\fE')$.

Also, one should note that the scheme of Higgs reductions is in general singular,
so that in order to consider  Higgs bundles on it one needs to use the theory
of the de Rham complex for arbitrary schemes, as developed by Grothendieck
\cite{EGAIV}.

For the time being we restrict our attention to the case when $X$
is a curve. We start by introducing a notion of semistability for
principal Higgs bundles (which is equivalent to the one given in
Definition 4.6 in \cite{AnBis}).

\begin{defin}\label{Hstabcurv} Let $X$ be a smooth projective curve. A principal Higgs $G$-bundle
$\fE=(E,\phi)$ is stable (resp.~semistable)  if for every parabolic subgroup $P\subset G$
and every Higgs reduction $\sigma\colon X\to {\fR}_P(\fE)$ one has
$\deg \sigma^\ast (T_{E/P,X})> 0$ (resp.~$\deg \sigma^\ast (T_{E/P,X})\ge 0$). \end{defin}

\begin{lemma} Let $f\colon X' \to X$ be a nonconstant morphism
of smooth projective curves, and $\fE$ a principal Higgs $G$-bundle on $X$.
The pullback Higgs bundle $f^\ast\fE$ is semistable if and only if
$\fE$ is.
\label{basechange}
\end{lemma}
\begin{proof} As we shall prove in Lemma \ref{adjsemi}
in the case of $X$ of arbitrary
dimension, a principal Higgs bundle $\fE$ is semistable if and only if the
adjoint Higgs bundle $\Ad(\fE)$ is semistable (as a Higgs vector bundle).
In view of this result, our claim reduces to the analogous statement
for Higgs vector bundles, which was proved in \cite{BH}.
\end{proof}

If $\fE=(E,\phi)$ is a principal Higgs $G$-bundle on $X$, and
$K$ is a closed subgroup of $G$, we may associate with every character $\chi$ of $K$
a line bundle $L_\chi=E\times_\chi \C$ on $E/K$, where we regard $E$ as a principal
$K$-bundle on $E/K$. An elegant way to state results about reductions is to introduce the notion
of \emph{slope} of a reduction: we call $\mu_\sigma$, the slope of a  Higgs reduction $\sigma$,
the group homomorphism $\mu_\sigma\colon \mathcal X(K) \to \Q$ (where $ \mathcal X(K) $
is the group of characters of $K$) which to any character $\chi$ associates the degree
of the line bundle $\sigma^\ast(L_\chi^\ast)$.

By a simple modification of the proof of  Lemma 2.1 of \cite{Rama}
we can extend it  to Higgs bundles.
If $\g$ is the Lie algebra
of $G$ and $\g'=[\g,\g]$ is its semisimple part, let $\alpha_1,\dots,\alpha_r$ be simple roots of
$\g'$, and let $\lambda_1,\dots,\lambda_r$
be the corresponding system of fundamental weights of $\g'$. Given a parabolic subgroup
$P\subset G$, a character $\chi\colon P\to\C^\ast$ is said to be
\emph{dominant} if it is a linear combination of the fundamental weights $\lambda_i$
with nonnegative coefficients.
Such a character is  trivial on the centre $Z(G)$ of $G$.

\begin{lemma} A principal Higgs $G$-bundle
$\fE=(E,\phi)$ is semistable  if and only if for every parabolic subgroup $P\subset G$,
every nontrivial dominant character $\chi$ of $P$, and every Higgs reduction $\sigma\colon X\to {\fR}_P(\fE)$, one has $\mu_\sigma(\chi)\ge 0$.
\label{RamaHiggs}\end{lemma}
\begin{proof}   We may  at first assume that $P$ is a maximal parabolic
subgroup corresponding to a root $\alpha_i$. It has been proven in
\cite[Lemma 2.1]{Rama} that the determinant of the vertical tangent bundle
$T_{E/P,X}$ is associated to the principal $P$-bundle $ E \to E/P$
via a character that may be expressed as $\mu = - m\lambda_i$, where
$\lambda_i$ is the weight corresponding to $\alpha_i$, and $m\ge 0$.
Thus, if $\sigma \colon X \to {\fR}_P(\fE)$ is a Higgs reduction,
 $\deg(\sigma^\ast(L^\ast_\mu))\ge 0 $ if and only if $\deg \sigma^\ast (T_{E/P,X})\ge 0$.
 
 If $P$ is not maximal, any dominant character of $P$ is a sum of
 dominant characters $\chi_k$ of the maximal parabolic subgroups $P_k$ that contain $P$,
 with $k=1,\dots,m$ for some $m$.
 Moreover, any Higgs reduction
 $\sigma\colon X\to\fR_P(\fE)$ induces a Higgs reduction $\sigma_k \colon X \to
 \fR_{P_k}(\fE)$. If $\fE$ is semistable, we have $\deg\sigma_k^\ast(L_{\chi_k})^\ast\ge 0$.
 Since  $\sigma^\ast(L_{\chi})\simeq \sigma_1^\ast (L_{\chi_1}) \otimes \dots  \otimes  \sigma_m^\ast (L_{\chi_m})$, we have $\mu_\sigma(\chi)\ge 0$.
 \end{proof}

We may now state and prove a Miyaoka-type semistability criterion
for principal Higgs bundles over projective curves. This generalizes Proposition 2.1
of \cite{BB1}, and, of course, Miyaoka's original criterion in \cite{Mi}.

\begin{thm} A principal Higgs $G$-bundle
$\fE=(E,\phi)$ on a smooth projective curve $X$ is semistable if and only if for every parabolic subgroup $P\subset G$, and
every nontrivial dominant character $\chi$ of $P$, the line bundle $L^\ast_\chi$ restricted to
${\fR}_P(\fE)$ is nef. \label{princMiya}
\end{thm}
\begin{proof} Assume that $\fE$ is semistable and that ${L^\ast_\chi}_{\vert {\fR}_P(\fE)}$
is not nef. Then there is an irreducible curve $Y\subset {\fR}_P(\fE)$ such that
$[Y]\cdot c_1(L^\ast_\chi)<0$. Since $\chi$ is dominant, the line bundle $L^\ast_\chi$ is nef when restricted
to a fibre of the projection $E/P\to X$, so that the curve $Y$ cannot be contained in such a fibre.
Then $Y$ surjects onto $X$. One can choose  a morphism of  smooth
projective curves  $h\colon X'\to X$   such that $\tilde Y = X' \times_X Y$ is a  curve in $h^\ast({\fR}_P(\fE))$, whose irreducible components are smooth and 
map   isomorphically to $X'$ (i.e., $\tilde Y \to X'$ is a split unramified covering). By
Lemma \ref{basechange},  the pullback of $\fE$ to $\tilde Y$ is semistable.
We may think of the irreducible components of $\tilde Y$ as images of sections $\sigma_j$ of $h^\ast({\fR}_P(\fE))$. 
By Lemma \ref{RamaHiggs}  this implies that $\deg \sigma_j^\ast (L')^\ast\ge 0$,
where $L'$ is the pullback of $L_\chi$ to $h^\ast({\fR}_P(\fE))$. This in turn implies
$[Y]\cdot c_1(L^\ast_\chi)\ge 0$, but this contradicts our assumption.

The converse is obvious in view of Lemma \ref{RamaHiggs}.
\end{proof}

\begin{remark} \label{remframe} Let  $G$ be the linear group $Gl(n,\C)$.
 If $\fE=(E,\phi)$ is a principal
Higgs $G$-bundle, and $V$ is  the rank $n$ vector bundle corresponding to $E$,
then the identification $\Ad(E)\simeq\End(V)$ makes $\phi$ into a Higgs
morphism $\tilde\phi$  for $V$. The semistability
of $\fE$ is equivalent to the semistability of the Higgs vector bundle
$(V,\tilde\phi)$.

If $P_k$ is   a maximal parabolic subgroup of $Gl(n,\C)$,
$E/P_k$ is the Grassmann bundle $\operatorname{Gr}_k(V)$ of rank $k$ locally free
quotients of $V$.
Then Theorem \ref{princMiya} corresponds to the result given in \cite{BH},
according to which $(V,\phi)$ is semistable if and only if certain numerical classes
$\theta_k$ in a closed subscheme of $\operatorname{Gr}_k(V)$ are nef (see the section \ref{appendix} and  \cite{BH,BG1,BG2} for details).
\end{remark}

\bigskip\section{The higher-dimensional case}
In this section we consider the case of a base variety $X$ which is
a complex projective manifold of any dimension. Let $X$ be equipped with a polarization $H$, and let $G$ be a reductive
complex algebraic group.

\begin{defin} \label{Higgsstabhigher} A principal Higgs $G$-bundle $\fE=(E,\phi)$ is stable (resp.~semistable)
if and only if for any proper parabolic subgroup $P\subset G$,   any open dense subset $U\subset X$ such that
$\operatorname{codim}(X-U)\ge 2$, and any Higgs reduction $\sigma\colon U \to {{\fR}_P(\fE)}_{\vert U}$ of $G$ to $P$ on $U$,
one has $\deg \sigma^\ast (T_{E/P,X}) > 0$ (resp.~$\deg \sigma^\ast (T_{E/P,X}) \ge 0$). \end{defin}

\begin{remark} The arguments in the proof of Lemma \ref{RamaHiggs}
go through also in the higher dimensional case, allowing one to show that
a principal Higgs $G$-bundle $\fE$ is semistable (stable) --- according to Definition
\ref{Higgsstabhigher} --- if and only if for any proper parabolic subgroup $P\subset G$, any nontrivial dominant character $\chi$ of $P$, any open dense subset $U\subset X$ such that
$\operatorname{codim}(X-U)\ge 2$, and any Higgs reduction $\sigma\colon U \to {{\fR}_P(\fE)}_{\vert U}$ of $G$ to $P$ on $U$,
the line bundle $\sigma^\ast( L_{\chi}^\ast)$ has nonnegative (positive) degree.
\label{remRama}
 \end{remark}

It is known that certain  extensions of the structure group of a
semistable principal bundle are still semistable \cite{RamaRama},
and that a principal bundle is semistable if and only if its
adjoint bundle is \cite{Rama}. The same is true in the Higgs case.

\begin{lemma} \label{adjsemi} (i) A principal Higgs bundle
$\fE$ is semistable if and only if $\Ad(\fE)$ is semistable (as a Higgs vector bundle).

(ii) A principal Higgs $G$-bundle $\fE=(E,\phi)$ is semistable if and only if
for every  linear representation $\rho\colon G\to\operatorname{Aut}(V)$ of $G$ such that
$\rho(Z(G)_0)$ is contained in the centre of $\operatorname{Aut}(V)$, the associated Higgs vector
bundle $\mathfrak V = \fE\times_\rho V$ is semistable (here $Z(G)_0$ is the
component of the centre of $G$ containing the identity).
\end{lemma}
\begin{remark} If $G$ is the general linear group $Gl(n,\C)$,
the first claim  holds true quite trivially: $\fE$ is semistable if and only if the corresponding
Higgs vector bundle $\mathfrak V$ is semistable, and one knows that $\Ad(\fE)\simeq\End(\mathfrak V)$
is semistable if and only if $\mathfrak V$ is.\end{remark}
\begin{proof}
The first claim is Lemma 4.7 of \cite{AnBis}. The second claim is proved as in
Lemma 1.3 of \cite{AAB}.
\end{proof}

\begin{prop} \label{semext}  Let $\lambda\colon G\to G'$ be a homomorphism
of connected reductive algebraic groups which maps the connected component
of the centre of $G$ into the connected component of the centre of $G'$.
If $\fE$ is a semistable principal Higgs $G$-bundle, and $\fE'$ is
obtained by extending the structure group $G$ to $G'$ by $\lambda$, then $\fE'$ is semistable.
\end{prop}
\begin{proof} By composing the adjoint representation of $G'$ with the homomorphism
$\lambda$ we obtain a representation $\rho\colon G\to \operatorname{Aut}(\mathfrak g')$; the principal Higgs bundle
obtained by extending the structure group of $\fE$ to $\operatorname{Aut}(\mathfrak g')$ is the bundle of linear
frames of $\Ad(\fE')$ with its natural Higgs field. By Lemma \ref{adjsemi}, this
bundle is semistable, so that $\Ad(\fE')$ is semistable as well. Again by Lemma  \ref{adjsemi},
$\fE'$ is semistable.
\end{proof}

\begin{remark} \label{remSimp} 
A notion of semistability for principal Higgs bundles
was introduced by Simpson in \cite{Si1}.  Let us say that 
a principal Higgs $G$-bundle $\fE$ is \emph{Simpson-semistable} if there exists
a faithful linear representation $\rho\colon G\to\operatorname{Aut}(W)$
such that the associated Higgs vector bundle $\mathfrak W = \fE\times_\rho W$
is semistable.  It is not difficult to show that Simpson-semistability implies
semistability; indeed if $\fE$ is Simpson-semistable, and $\rho$
is a faithful linear representation such that $\mathfrak W$ is semistable,
then $\End(\mathfrak W)$, with its natural Higgs bundle structure, is semistable.
But $\End(\mathfrak W)\simeq \operatorname{Ad}(GL(\mathfrak W))$,
and $ \operatorname{Ad}(E)$ is a subbundle of $\operatorname{Ad}(GL(\mathfrak W))$.
Since both $ \operatorname{Ad}(E)$ and $\operatorname{Ad}(GL(\mathfrak W))$
have vanishing first Chern class, $ \operatorname{Ad}(E)$ is semistable, so that
$E$ is semistable as well.

The contrary is not true, even in the case
of ordinary (non-Higgs) principal bundles (in which case of course our definition coincides with Ramanathan's classical definition of stability for principal bundles \cite{Rama}). Indeed,  if $T$ is a torus in $Gl(n,\C)$, any principal $T$-bundle $E$ is stable. However the vector bundle
associated to it by the natural inclusion $T\hookrightarrow Gl(n,\C)$ (a direct sum of line bundles)
may fail to be semistable, in which case $E$ cannot be Simpson-semistable. (Note indeed that this inclusion, regarded as a linear representation
of $T$, does not satisfy the condition in part (ii) of Lemma \ref{adjsemi} unless $n=1$.) A point in favour of  the definition we choose is that it  is compatible with the Hitchin-Kobayashi correspondence for principal bundles, which states
that a principal $G$-bundle $E$, where $G$ is a connected reductive complex group, is polystable
if and only if it admits a reduction of the structure group to the maximal compact subgroup $K$
of $G$ such that the mean curvature of the  unique connection on $E$ compatible with the reduction takes
values in the centre of the Lie algebra of $K$ \cite{RamaSub}.
(We shall recall the definition of polystability of a principal Higgs bundle in section \ref{flatcon}.)
\end{remark}

We can now prove a version of Miyaoka's semistability criterion
which works for principal Higgs bundles on projective varieties of
any dimension.

\begin{thm}\label{Miyahigher} Let $\fE$ be a  principal Higgs $G$-bundle $\fE=(E,\phi)$ on $X$. Consider the following conditions:
\begin{enumerate} \item
for every parabolic subgroup $P\subset G$ and any nontrivial dominant character $\chi$ of $P$, the line bundle $L_{\chi}^\ast$ restricted to ${\fR}_P(\fE)$ is numerically effective;
\item for every morphism $f\colon C\to X$, where $C$ is a smooth projective curve,  the pullback
 $f^\ast(\fE)$   is semistable.
\item $\fE$ is semistable and $c_2(\Ad(E))=0$ in $H^4(X,\R)$.
\end{enumerate}
Then conditions (i) and (ii) are equivalent, and they are both implied by condition (iii).
\end{thm}
\begin{proof}
Assume that condition (i) holds, and let $f\colon C\to X$ be as in the statement.
The line bundle $L'_\chi$ on $f^\ast(E)/P$ given by the character $\chi$
is a pullback of $L_\chi$. Then $L'_{\chi\vert  {\fR}_P(f^\ast\fE)}$ is nef,
so that by Theorem \ref{princMiya},  $f^\ast(\fE)$  is semistable. Thus (i) implies (ii).

We show  now that (ii) implies (i). Let $C'$ be a curve in
${\fR}_P(\fE)$. If it is contained in a fibre of the projection $\pi_P\colon {\fR}_P(\fE)\to X$,
since $\chi$ is dominant, we have $c_1(L_\chi^\ast)\cdot[C']\ge 0$. So we
may assume that $C'$ is not in a fibre.    The projection
of $C'$ to $X$ is a finite cover  $\pi_P\colon C'\to C$ to its
image $C$. We may choose a smooth projective curve $C''$ and a
morphism $h\colon C''\to C$ such that $\tilde C=C''\times_CC'$ is
a split unramified covering. Then every sheet $C_j$ of $\tilde C$ is
the image of a section $\sigma_j$ of ${\fR}_P(h^\ast\fE)$. Since
$h^\ast\fE$ is semistable by Lemma \ref{basechange}, we have $\deg
\sigma_j^\ast (L^\ast_\chi)\ge 0$ by Lemma \ref{RamaHiggs}. This
implies (i).

Finally, we prove that (iii) implies (ii). $\Ad(\fE)$ is semistable by  Lemma \ref{adjsemi}; thus,  since $c_2(\Ad(E))=0$, by Theorem \ref{1.3BHR} the Higgs vector bundle
$\Ad(f^\ast(\fE))$ is semistable, and then $f^\ast(\fE)$  is semistable
by Lemma \ref{adjsemi}.
\end{proof}

\begin{remark} \label{cosiva} For non-Higgs principal bundles, one actually proves that condition
(iii) in Theorem \ref{Miyahigher} is equivalent to conditions (i) and (ii)
\cite{BB1}.
\end{remark}

\begin{corol} Assume that $\fE=(E,\phi)$ is a principal Higgs $G$-bundle,
$\lambda\colon G\to G'$ is a surjective group homomorphism, $\fE'=(E',\phi')$ is a principal Higgs $G'$-bundle, and $f\colon E\to E'$ is a $\lambda$-equivariant morphism of principal Higgs bundles.
If $\fE$ satisfies condition (i) or (ii) of Theorem \ref{Miyahigher}, so does $\fE'$.
\end{corol}
\begin{proof} If $P'$ is a parabolic subgroup of $G'$, then $P'=\lambda(P)$ for
a parabolic $P$ in $G$. If $\chi'\colon P'\to \C^\ast$ is a dominant character of $P'$,
the composition $\chi=\chi'\circ\lambda$ is a dominant character of $P$. If $f\colon E/P\to E'/P'$
is the induced morphism, we know that $f({\fR}_P(\fE))\subset {\fR}_{P'}(\fE')$, so that
$f^\ast(L^\ast_{\chi'\vert {\fR}_{P'}(\fE')})\simeq L^\ast_{\chi\vert {\fR}_P(\fE)}$. Since
$L^\ast_{\chi\vert {\fR}_P(\fE)}$ is nef, and $f\colon  {\fR}_P(\fE)\to {\fR}_{P'}(\fE')$ is surjective,
$L^\ast_{\chi'\vert {\fR}_{P'}(\fE')}$ is nef as well \cite{Fu}.
\end{proof}

In \cite{BG1} we introduced a notion of numerically flat Higgs vector
bundle (see also Section \ref{appendix} of this paper). A special class of semistable principal Higgs bundles provides examples of such bundles.

\begin{thm}\label{AdHnflat} Let $\fE$ be a  principal Higgs bundle $\fE=(E,\phi)$ on a polarized smooth complex projective variety $X$.  If $\fE$ is semistable and $c_2(\Ad(E))=0$ in $H^4(X,\R)$, then 
 the adjoint Higgs bundle $\Ad(\fE)$ is H-nflat.
 \end{thm}
\begin{proof} At first we prove this theorem when $X$ is a curve. In this case actually we can prove that   $\fE$ is semistable
if and only if $\Ad(\fE)$ is H-nflat. In view of Lemma \ref{adjsemi}, this amounts to proving that $\Ad(\fE)$ is semistable
if and only if it is H-nflat. Since $c_1(\Ad(E))=0$ this holds true 
(Lemma \ref{curverestrHnef} and Proposition \ref{flatisss}, see also
\cite{BG1}, Corollaries 3.4 and 3.6).

Let us assume now that $\dim(X)>1$. If condition (i) holds, then $\fE_{\vert C}$ is semistable for any embedded curve
$C$ (as usual, if $C$ is not smooth one replaces it with its normalization). Thus $\Ad(\fE)_{\vert C}$ is
semistable, hence H-nflat. But this implies that $\Ad(\fE)$ is H-nflat as well.
\end{proof}

 \begin{remark} \label{remBB} For non-Higgs principal bundles, one is able to prove that
 the two conditions in the statement of Theorem \ref {AdHnflat} are equivalent \cite{BB2}.
  This characterization shows that the numerically flat principal $G$-bundles defined in \cite{BS} for semisimple structure groups $G$ are no more than the class of principal bundles singled out by one of the   conditions of Theorem \ref{Miyahigher}; cf.~\cite[Thm.~2.5]{BS},
and Propositions \ref{nflatss} and \ref{Hconv}.
 \end{remark}

\bigskip\section{Numerically effective principal (Higgs) bundles}\label{Hnflat}
In this section we wish to give a definition of numerical effectiveness and numerical flatness for principal (Higgs)  bundles on a complex projective manifold $X$, and prove its main properties.

We start with some group-theoretic considerations. Given  a complex reductive algebraic group $G$, let $P\subset G$ be a  parabolic subgroup, and $R_u(P)$ the unipotent radical of $P$. A subgroup $L$ of $G$ such that $L\simeq P/R_u(P)$, and $P$
is a semidirect product $P=LR_u(P)$, is called a {\em Levi factor} of $P$. All Levi factors
are conjugated by elements of $R_u(P)$, and are reductive algebraic groups,
whose root system is in general reducible; hence a Levi factor $L$ may be written
as $L=L_1\cdots L_m$ according to the decomposition of its root system  \cite[Sect.~27.5]{Hump}. 

Now let $\rho\colon G\to Gl(V)$ be a faithful rational representation,
let $W$ be a subspace of $V$, and let $P$ be a  maximal parabolic subgroup 
of $G$ which stabilizes $W$. There is an induced action of $P$ on $V/W$.
 A factor $L_i$ of the Levi group
of $P$ is said to be a \emph{standard quotient} of $P$ if
$\rho$ maps it injectively into $Gl(V/W)$ for some choice of $\rho$ and $W$.

We may now define a notion of universal quotient bundle of  a principal Higgs bundle.
Let $\fE=(E,\phi)$ be a principal Higgs $G$-bundle on a projective manifold $X$. For any closed subgroup
$K\subset G$, denote by $E_K$ the principal $K$-bundle $E \to E/K$.
(Recall that the restriction of $E_K$ to the scheme of   Higgs reductions $\fR_{K}\subset E/K$
carries an induced Higgs field, cf.~Remark \ref{inducedHiggs}, thus giving rise
to a principal Higgs $K$-bundle $\fE_K$). 
If $P\subset G$ is a parabolic subgroup, and $\psi\colon P\to Q$
the projection onto a standard quotient, we call $E_Q$ the principal $Q$-bundle obtained by extending the structure group of $E_P$ to $Q$.

\begin{defin} A universal Higgs quotient   $\fE_Q$ of $\fE$ is the restriction
of $E_Q$ to the scheme of Higgs reductions $\fR_{P}(\fE)\subset E/P$,
equipped with the Higgs field induced by the Higgs field of $\fE_{P}$.
Here $P$ is a maximal parabolic subgroup of $G$, and $Q$ is a standard quotient of $P$.
\end{defin}

\begin{remark}
\label{motiv} The motivation for this definition is as follows. If $G$ is the general linear group 
$Gl(V)$, where $V$ is a complex finite-dimensional vector space, a maximal parabolic subgroup $P$ in $G$    stabilizes   a subspace $W\subset V$.  Then 
a standard quotient of $P$ is isomorphic to the group $Gl(V/W)$. If $U$ is a vector bundle on a variety $X$, and $E$ is the bundle
of linear frames of $U$,    the principal $Q$-bundle obtained by extending the structure group of $E_P$ to $Q$ is the bundle of linear frames of the universal rank $k$ quotient bundle on the Grassmannian bundle $E/P$, where $k=\dim(V/W)$.
\end{remark}

\begin{remark}\label{funct} Note that this construction is functorial: if $f\colon Y\to X$ is a morphism
of projective manifolds, then $(f^\ast\fE)_Q\simeq\bar f^\ast \fE_Q$, where
$\bar f\colon\fR_{P}(f^\ast\fE) \to \fR_{P}(\fE) $ is the morphism induced by $f$.
\end{remark}

We give now our definition of numerical effectiveness. This will be a recursive
definition, with recursion on the semisimple rank of the structure group, and we start by defining numerical effectiveness for what will be the ``terminal" case,
i.e., principal Higgs $T$-bundles, where $T$ is an algebraic torus. 

\begin{defin} \label{neftori} Let $\fE=(E,\phi)$ be a principal Higgs $T$-bundle, with $\dim T = r$.
\begin{enumerate} \item  $\fE$  is \emph{Higgs-numerically effective} (H-nef for short) if  there exists an isomorphism $\lambda\colon T \to (\C^\ast)^r$
such that the vector bundle associated to $E$ via $\lambda$ is nef.
\item $\fE$ is  \emph{Higgs-numerically flat} (H-nflat for short) if  there exists an isomorphism $\lambda \colon T \to (\C^\ast)^r$
such that the vector bundle $V_\lambda$ associated to $E$ via $\lambda$ is numerically flat, i.e.,
both $V_\lambda$ and $V_\lambda^\ast$ are numerically effective.
\end{enumerate}
\end{defin}
Higgs-numerical flatness can be equivalently   defined by asking that the vector bundle associated to $E$ via
\emph{any} isomorphism $ T \to (\C^\ast)^r$ is numerically flat. Note that these definitions are independent
of the Higgs field.

 Let $D(G)$ be the derived subgroup of $G$. The quotient
 $R'=G/D(G)$ is isomorphic to the quotient of the radical $R$ of $G$  by  a finite subgroup, and is  therefore isomorphic to $R$. 
Let $\rad\colon G \to  R$ be the   projection.
\begin{defin} The radical  of a principal Higgs $G$-bundle $\fE$ is the principal Higgs $R$-bundle $R(\fE)=\fE\times_{\rad}R\simeq \fE/D(G)$.
\end{defin}
If $\fE$ is the bundle of linear frames of a Higgs vector bundle $\mathfrak V=(V,\phi)$,
then  $R(\fE)$ is the bundle of linear frames of the determinant line bundle  $\det(V)$ equipped with the induced  Higgs field $\det(\phi)$.

\begin{prop} The radical   $R(\fE)$ of a principal Higgs $G$-bundle $\fE$ is trivial
(as a principal Higgs bundle, see Definition \ref{trivial})
if and only if $\fE$ admits a Higgs reduction of its structure group to its derived subgroup
$D(G)$.  
\end{prop}
\begin{proof} If $R(\fE)$ is trivial, the principal $R$-bundle $E/D(G)$ is trivial,
so that the structure group of $E$ may be reduced to $D(G)$; let us denote by $E'$
the reduced bundle.  Since the Higgs field
of  $R(\fE)$ is zero, the Higgs field $\phi$ of $\fE$ is actually a section of $\Ad(E')\otimes\Omega_X^1$,
so that $\fE'=(E',\phi)$ is a Higgs reduction of the structure group of
$\fE$ to $D(G)$.

Conversely, if such a reduction exists, $R(E)=E/D(G)$ is trivial as it has a global section, and since 
$\phi$ lies in $\Gamma(\Ad(E')\otimes\Omega_X^1)$, the Higgs field of $R(\fE)$ vanishes.
\end{proof}

\begin{defin} \label{P-Hnef} A principal Higgs $G$-bundle $\mathfrak E$ on $X$ is H-nef  if

\begin{enumerate}\item  $R(\fE)$ is H-nef  according to Definition \ref{neftori};

\item  if  $\rk_{ss}(G)>0$, for every   maximal parabolic subgroup $P$ 
and every standard quotient $Q$ of $P$, the 
universal Higgs quotient $\fE_Q$ is H-nef.
\end{enumerate}

Moreover, $\fE$ is said to be H-nflat  if it
is H-nef and $R(\fE)$ is H-nflat. \end{defin}

Since the semisimple rank of the structure group $Q$ of $\fE_Q$ is strictly smaller
than the semisimple rank of $G$, this recursive definition makes sense.
As far as we know, this definition is new even in the case of (non-Higgs) principal bundles.

\begin{remark} \label{tworem} (i) If $\fE$ is the bundle of linear frames of a Higgs vector bundle
$\mathfrak V$, then, in view of Remark \ref{motiv}, it is H-nef (H-nflat) if and only if 
$\mathfrak V$ is H-nef (H-nflat) in the sense of Definition \ref{moddef}. As a further particular case,
when  the Higgs field is zero, so that we are dealing with an ordinary principal $Gl(n,\C)$-bundle, the latter is nef in this sense if and only if the associated
vector bundle is nef in the usual way.

(ii) Definition \ref{P-Hnef} implies that a principal Higgs $G$-bundle is H-nef 
if and only if $f^\ast\fE$ is H-nef for all morphisms $f\colon C \to X$ where $C$ is a smooth algebraic
curve. 
\end{remark}

We prove some basic properties of H-nef principal Higgs bundles.

\begin{prop}\label{prop} {\rm (i)} The pullback of an H-nef principal Higgs bundle is H-nef.

{\rm (ii)} A trivial   Higgs  $G$-bundle is H-nflat.
\end{prop}

\begin{proof} Point (i) follows immediately from Remark \ref{funct}, or from Remark
\ref{tworem}(ii). The proof  of point (ii)  needs   the following preliminary result.

Let $G$ be a reductive linear algebraic group, $P\subset G$ a
maximal parabolic subgroup, and let $E_G$ be the   principal $G$ bundle
over $G/P$ obtained by extending the structure group of the principal $P$-bundle
$G\to G/P$ to $G$ via the inclusion $P\to G$.  One easily checks that $E_G$ is trivial. Let $\fE_G$ be $E_G$ equipped with the trivial Higgs field.
Then $\fE_G$ is H-nef.

 We prove this by induction the semisimple rank of  $G$. If
$ \rk_{ss}(G)=0$, then  $\fE_G$  is the bundle of linear frames of a trivial Higgs vector bundle on $G/P$, so that it is H-nef (cf.~Remark \ref{tworem}(i)).

If $ \rk_{ss}(G)>0$, 
we first prove that $R(\fE_G)$ is H-nef. Let $\chi\colon R(G)\to\C^\ast$ be a character
of the radical of $G$. The associated Higgs $\C^\ast$-bundle is trivial,
hence H-nef by Remark \ref{tworem}(i), and then $R(\fE_G)$ is H-nef.   

The inductive step is used to prove that the universal Higgs quotients of  $\fE_G$ are
H-nef. Let $P'\subset G$ be a maximal parabolic subgroup of $G$, and let $\psi'\colon
P'\to Q'$ be the projection onto a standard quotient. The associated universal
principal Higgs quotient is the pullback of the universal quotient $G\times_{\psi'}Q'$
via the projection $G/P\times G/P'\to G/P'$ (with the zero Higgs field). Now, $G\times_{\psi'}Q'$ is H-nef
by the inductive hypothesis, and its pullback is H-nef due to point (i) of this
Proposition. So we have proved the inductive step.

Now we go back to the proof of point (ii).  If $\fE=X\times G\to X$ with trivial Higgs field, then
$R(\fE)\simeq X \times R(G)$ is the bundle of linear frames of
a trivial Higgs vector bundle on $X$, so that it is H-nflat. Moreover,
let $P\subset G$ be a parabolic subgroup, and $Q$ its standard quotient.
Then the associated universal quotient of $\fE_G$ is the pullback
of the universal quotient of the bundle $G \to G/P$ via the projection $X\times G/P\to G/P$,
hence is H-nef due to point (i) and to the result we have previously proved. Thus $\fE$ is H-nef, and since $R(\fE)$ is H-nflat, $\fE$ is also H-nflat.
\end{proof}

Numerically flat principal Higgs bundles turn out to be semistable.

\begin{prop} \label{nflatss} An H-nflat principal Higgs $G$-bundle $\fE$ is semistable.\end{prop}

\begin{proof} 
Let   $P\subset G$  be a maximal
parabolic subgroup, and  $\chi$   a nontrivial dominant character  of $P$. Let $Q$ be a standard quotient
of $P$, and $\psi\colon P\to Q$ the projection.   Given a character $\chi_Q\colon Q\to\C^\ast$ we may define
a character $\chi'$ of $P$ by letting $\chi'=\chi_Q\circ\psi$.

Since  the  universal quotient $\fE_Q$ is an H-nef  principal Higgs $Q$-bundle, 
the radical bundles $R(\fE_Q)$ are H-nef as well, and we may choose
the  character  $\chi_Q\colon Q\to\C^\ast$ so that the restriction of the dual of the line bundle
$L_Q=E_Q\times_{\chi_Q}\C$ to $\fR_P(\fE)\subset E/P$ is 
nef (cf.~Definition \ref{neftori}: $\chi_Q$ may be taken as the composition of the iso morphism $\lambda$ with the determinant morphism $(\C^\ast)^r$). 
Let $L'$ be the line bundle on $E/P$ associated to $E_P$ 
by the character $\chi'$. One defines a morphism 
\begin{eqnarray*} L' & \to &  L_{Q} \\
(g,z) &\mapsto & ((g,e),z) \end{eqnarray*}
 which turns out to be surjective,
hence it is an isomorphism. Since $\operatorname{Pic}(G/P)\simeq \Z$, we have $m_1 \chi = m_2\,\chi'+\chi_0$,
   for some integers $m_1$, $m_2$ and a character $\chi_0$ of the centre of
$G$. The line bundle $L^\ast_\chi$ is nef when restricted
to the fibres of $E/P\to X$ (which are copies of $G/P$), while
$L_Q^\ast$ is nef after restricting to the intersections
of these fibres with the scheme of Higgs reductions $\fR_P(\fE)$, and
  the restriction of the line bundle associated to $\chi_0$ is numerically flat.
Hence we may assume that $m_1$ and $m_2$ are both positive. Therefore
$L^\ast_{\chi\vert \fR_P(\fE)}$ is nef. 
This by Theorem \ref{Miyahigher} implies the claim.
\end{proof}

\begin{remark} For non-Higgs principal bundles,
one can prove that a principal $G$-bundle $E$ is numerically flat if and only
if it is semistable and  $c_2(\Ad(E))=0$.
\end{remark}

 \begin{prop}\label{Hconv} If a principal Higgs $G$-bundle $\fE$ is semistable, satisfies\break
$c_2(\Ad(E))  =0$, and its radical $R(\fE)$ is H-nflat, then it is H-nflat.
\end{prop}
\begin{proof} 
Since by hypothesis $R(\fE)$ is H-nflat, 
we   only need to show that all universal quotient principal Higgs bundles
$\fE_Q$ are H-nef. In particular, in virtue of our recursive definition, we   need to show that all radicals $R(\fE_Q)$ are
H-nef, and that a number of other radicals are H-nef as well (cf.~Proposition \ref{crit}). Let us just check why
the radicals $R(\fE_Q)$ are
H-nef. Now, it turns out that  every character
of $Q$ composed with the projection $\psi_s\colon P\to Q$ is 
a (possibly rational) multiple of a dominant character $\chi$ of $P$. 
Since $\fE$ is semistable, and  $c_2(\Ad(E))  =0$,
by Theorem \ref{Miyahigher}, the line bundle $L_\chi^\ast$ is nef.
This implies the existence of an isomorphism $R(Q)\stackrel{\sim}{\to} (\C^\ast)^r$
such that the vector bundle associated by it to $R(\fE_Q)$ is nef.
This means that
$R(\fE_Q)$ is H-nef.  \end{proof}

\bigskip\section{Numerically flat principal bundles and flat reductions}\label{flatcon}
In \cite{DPS94} numerically flat vector bundles were characterized as vector bundles admitting filtrations
whose quotients are locally free and stable, and admit flat unitary connections. 
In this section we prove a similar result for principal bundles, with a partial generalization to principal Higgs bundles.

We start by reviewing some facts about connections on principal bundles,
covering also the case when a Higgs field is present. 
Let $K$ be a maximal compact subgroup of a
connected reductive complex algebraic group $G$. Note that the Lie algebra
$\g$ of $G$ admits an involution $\iota$, called the \emph{Cartan involution}, whose +1 eigenspace is the Lie algebra $\mathfrak k$ of $K$. If $\fE=(E,\phi)$ is a principal Higgs $G$-bundle,
we may extend $\iota$ to an involution on the sections of the bundle $\Ad(E)\otimes\mathcal A^1$
(where $\mathcal A^1$ is the bundle of complex-valued smooth differential  1-forms) by letting
$$\iota(s\otimes\omega) =- \iota(s)\otimes\bar\omega\,.$$

Given  a reduction $\sigma$ of the structure group of $E$ to $K$, there is a unique connection 
$\nabla_{\sigma}$ on $E$ which is compatible with the complex structure of $E$ and with the reduction \cite{RamaSub}. By analogy with the vector bundle case, we call it the \emph{Chern} connection associated with the reduction $\sigma$. The Higgs field may be used to introduce another connection $$\nabla_{\sigma,\phi} = \nabla_\sigma + \phi + \iota(\phi)$$
which we call the \emph{Hitchin-Simpson connection} of the triple $(\fE,\sigma)=(E,\phi,\sigma)$.

\begin{defin} A principal Higgs $G$-bundle $\fE$ is said to be \emph{Hermitian flat}
if it admits a reduction of its structure group to $K$ such that the corresponding  Hitchin-Simpson connection is flat.
\end{defin}

To state our results we need the notion of \emph{polystable} principal Higgs bundle.
Let us recall that the notion of \emph{slope} of a reduction was introduced
in section \ref{phb}, cf.~Lemma \ref{RamaHiggs}.

\begin{defin} \label{admi} 
A reduction $\sigma$ of the structure group of $G$ of a principal Higgs $G$-bundle
$\fE$ to a parabolic subgroup $P\subset G$  is said to be \emph{admissible} if $\mu_\sigma(\chi)=0$ for
every character of $\chi$ of $P$ which vanishes on the centre of $G$.
\end{defin}
\begin{defin} A principal Higgs $G$-bundle $\fE$ is said to be \emph{polystable}
if there is a parabolic subgroup $P$ of $G$ and a Higgs reduction $\sigma$ of the structure
group of $E$ to a Levi subgroup $L$ of $P$ such that
\begin{enumerate} \item
the reduced principal Higgs $L$-bundle $\fE_\sigma$ is stable;
\item the   principal Higgs $P$-bundle obtained by extending the structure
group of $\fE_\sigma$ to $P$ is an admissible reduction of the structure group of $\fE$ to
$P$ (cf.~Definition \ref{admi}). \end{enumerate}\end{defin}

Also in this case one has a \emph{Hitchin-Kobayashi correspondence} \cite{RamaSub}.
Choose a K\"ahler form $\omega$ on $X$ representing the polarization $H$ we are using. We 
say that a reduction $\sigma$ of the structure group $G$ of a principal Higgs $G$-bundle $\fE$ to a maximal compact  subgroup $K$ is \emph{Hermitian-Yang-Mills}  if there is an element $\tau$ in the centre
$\mathfrak z$ of the Lie algebra $\g$ of $G$ such that
$$ \mathcal K_{\sigma,\phi} = \tau$$
where $ \mathcal K_{\sigma,\phi} $ is the \emph{mean curvature} of the Hitchin-Simpson connection
(computed with the K\"ahler form $\omega$).

\begin{thm} \cite{AnBis} \label{HK} A principal Higgs $G$-bundle $\fE$ is polystable if and only
if it admits an Hermitian-Yang-Mills reduction to a maximal compact subgroup $K\subset G$.
\end{thm}
 
This notion of polystability extends the one holding for Higgs vector bundles,
i.e., a Higgs vector bundle is polystable if it is a direct sum of stable Higgs vector bundles having the
same slope. A result similar to Lemma  \ref{adjsemi}(i) may be proved. The proof of this
result is implicitly contained in \cite{AnBis}.

 \begin{prop} 
A principal Higgs bundle is polystable if and only if its adjoint bundle is polystable.
\label{polyrem}
\end{prop}

We state now our second main result in the case of (non-Higgs) principal bundles.
\begin{thm} \label{coolreduction} A principal   $G$-bundle
 $E$ is nflat    if and only if there is   a
parabolic subgroup $P$ of $G$ and a   reduction $\sigma$ of the structure group
of $E$ to $P$ such that the principal   $L(P)$-bundle obtained
by extending the structure group of the reduced   bundle $E_P$ to the Levi factor $L(P)$
is Hermitian flat and polystable.
\end{thm}
\begin{proof}
The ``if" part   is quite easily proved. Since $E$ admits
a flat connection, we have $c_2(\Ad(E))=0$. Moreover,
$E$ is polystable, hence semistable. The radical $R(E)$ carries an induced flat
connection. Hence Proposition \ref{Hconv} implies that
$E$ is nflat.

Let us now prove the ``only if" part.
In view of Remark \ref{remBB}, we know that $\Ad(E)$ is nflat.
As   showed in \cite{DPS94}, this implies that it has a filtration  
\begin{equation}\label{cleverfiltr}  0 \subset {S}_0 \subset \dots \subset {S}_m= \Ad(E) \end{equation}
such that every quotient ${S}_{i+1}/{S}_{i}$ is locally free, flat and stable. 
The analysis made in \cite{AnBis} (see also \cite{BB2}) may be carried over to the present situation:
one shows that the filtration \eqref{cleverfiltr} has an odd number of terms,
and the middle term (say, ${S}_\ell$) is isomorphic to the adjont bundle $\Ad(F)$
of a reduction $F$ of $E$ whose structure group is a parabolic subgroup
$P$ of $G$. 

Let $E_L$ be the principal  
$L(P)$-bundle obtained by extending the structure group of $F$ to $L(P)$. It turns out
that $\Ad(E_L)$ is isomorphic to the quotient ${S}_\ell/{S}_{\ell -1}$.
Since the successive quotients of the filtration \eqref{cleverfiltr} are stable and flat, the bundle  $\Ad(E_L)$ is stable, and moreover,
all its Chern classes vanish \cite{DPS94}. The polystability of $\Ad(E_L)$ implies the polystability
of $E_L$  (see Proposition \ref{polyrem}). By Theorem \ref{HK}, $E_L$ admits a reduction
to the maximal compact subgroup of $L(P)$ such that the corresponding Chern connection satisfies the Hermitian-Yang-Mills condition. 

Now, the homomorphism
$$L \to \Aut(\mathfrak l) \times R(L)$$
given by the adjoint representation of $L=L(P)$, and the projection onto the radical $R(L)$,
gives a injective Lie algebra homomophism
\begin{equation}\label{inj} \mathfrak l \to \End(\mathfrak l) \oplus \mathfrak r(L).\end{equation}
Here $\mathfrak l$ and $\mathfrak r(L)$ are the Lie algebras of $L$ and $R(L)$, respectively.
Thus we have a   vector bundle $V= \Ad(E_L)\oplus W$ which
is associated to $E_L$, and by Lemma \ref{adjsemi} is semistable. Then $\deg(W)=\deg( \Ad(E_L))=0$. Moreover,  
$V$ satisfies $\Delta(V)=0$ because $\Delta(V)$ is a multiple of $c_2(\Ad(E_L))$.
On the other hand, by the same reason we have $\Delta(W)=0$. This implies $c_1(W)^2=0$.

The Hermitian-Yang-Mills connection on $E_L$ induces Hermitian-Yang-Mills connections
on $\Ad(E_L)$ and $W$. Lemma IV.4.12 of \cite{koba}  (with the conditions
$\deg(W)=0$, $c_1(W)^2=0$) implies that the connection on
$\mathfrak W$ is flat, and the same is true for $\Ad(E_L)$. Since the morphism
\eqref{inj} is injective, the Hermitian-Yang-Mills connection on $E_L$  is flat as well.
\end{proof}
 
\begin{corol} If  $E$ is nflat, the cohomology ring of $E$ with coefficients in $\R$ is trivial. \label{weil}
\end{corol}
\begin{proof}
This follows from the fact that the principal Higgs $G$-bundle obtained by extending
the structure group of $E_L$ to $G$ is isomorphic to $E$ as a topological bundle.
Keeping up with the notation of Theorem \ref{coolreduction}, let $E$ be
a principal $G$-bundle, $F$ a reduced bundle with structure group a parabolic
subgroup $P\subset G$, and $E_L$ the principal $L$-bundle obtained
by extending the structure group of $F$ to $L$ (here $L$ is the Levi group
corresponding to $P$). Moreover, let $E'$ be the $G$-bundle obtained
by extending the structure group of $E_L$ to $G$; so, $E'$ is the ``graded object"
corresponding to the reduction of $G$ to $P$. Let $\rho\colon G\to \operatorname{Aut}(W)$ be a faithful representation of $G$, and let $V=E\times_\rho W$ be the associated vector bundle. There exists a flag  $0=W_0\subset \dots\subset W_\ell=W$ which is preserved by  $\rho(P)$, such that the unipotent radical of $P$ acts trivially
on the quotients $W_i/W_{i-1}$. Thus $\rho(P)$
 is contained in a parabolic subgroup $P'$ of $\operatorname{Aut}(W)$, and $\rho(L)$ is cointained in a Levi subgroup $L'$ of $P'$. The graded module $V'$
of the filtration of $V$ corresponding to $P'$ is isomorphic to the associated bundle
$E'\times_\rho W$, and on the other hand it is topologically isomorphic to $V$. This implies that $E$ and $E'$ are topologically isomorphic.
\end{proof}

\begin{remark}
By Remarks \ref{cosiva} and \ref{remBB}, the "if" part of Theorem 
\ref{coolreduction} holds true also for principal Higgs bundles.
\end{remark}

\bigskip\section{Some Tannakian considerations}\label{tannaka}
In this section we place  Theorem \ref{Miyahigher}  into the framework
of Tannakian categories. We recall (see e.g.~\cite{DM})  that a neutral Tannakian category $\cat T$ over
a field $\Bbbk$
is a rigid abelian (associative and commutative) $\Bbbk$-linear tensor category such  that
\begin{enumerate}
\item for every
unit object 1 in $\cat T$, the endomorphism space $\operatorname{End}(1)$
is isomorphic to $\Bbbk$;
\item there is an exact faithful functor $\omega\colon \cat T\to \cat{Vect}_{\Bbbk}$,
called a {\em fibre functor.}
\end{enumerate}
Here $\operatorname{Vect}_{\Bbbk}$ is the category of vector spaces over $\Bbbk$. The standard
example of a neutral Tannakian category is the category $\cat{Rep}(G)_{\Bbbk}$ of  $\Bbbk$-linear representations of
an affine group scheme $G$. Indeed, any neutral Tannakian category can be represented
as $\cat{Rep}(G)_{\Bbbk}$ where $G$ is the automorphism group of the fibre functor $\omega$.
Let $\fE$ be a principal Higgs $G$-bundle
on a (say) complex projective manifold $X$.
For any finite-dimensional linear representation $\rho\colon G\to\operatorname{Aut}(W)$ let
$\mathfrak W = \fE\times_\rho W$ be the associated Higgs vector bundle. This correspondence
defines a \emph{$G$-torsor} on the category $\cat{Higgs}_X$ of Higgs vector bundles on $X$, i.e.,
a faithful and exact functor $\fE\colon\cat{Rep}(G)_{\Bbbk}\to\cat{Higgs}_X$ \cite{Si1}. In general,
this is not always compatible with semistability, i.e., $\fE(\rho,W)$ is not always semistable
even when $\fE$ is. In order to have that, we need to impose some conditions. For instance,
we may assume  that  every representation $\rho\colon G\to\operatorname{Aut}(W)$ maps the
connected component of the centre of $G$ containing the identity to the centre
of $\operatorname{Aut}(W)$ (this happens, e.g., when $G$ is semisimple). When this is true, we
say that $G$ is \emph{central}. 

Let $\cat{Higgs}_X^\Delta$ be the full subcategory of $\cat{Higgs}_X$ whose
objects $\mathfrak W$ are Higgs vector bundles such that
$f^\ast\mathfrak W$ is semistable for every morphism $f\colon C\to X$, where
$C$ is any smooth projective curve (in particular, such Higgs bundles are
semistable).  Since the tensor product of semistable Higgs bundles is semistable \cite{Si1}, $\cat{Higgs}_X^\Delta$ it is a tensor category.  However, it is not additive but only preadditive (i.e., every homomorphism set $\operatorname{Hom}(\mathfrak V,\mathfrak W)$
is an abelian group, and composition of morphisms is bilinear over the integers).
 Let $\cat{Higgs}_X^{\Delta,+}$ 
be its additive completion (see e.g.~\cite{Freyd}). We may now prove the following characterization. \begin{prop} Assume that $G$ is central. There is one-to-one correspondence between principal Higgs $G$-bundles $\fE$ satisfying condition (i) or (ii) 
of Theorem \ref{Miyahigher} and $G$-torsors on the category $\cat{Higgs}_X$ taking values in   $\cat{Higgs}_X^{\Delta,+}$. \end{prop}
\begin{proof} Given a principal Higgs $G$-bundle $\fE$ and a representation  $\rho\colon G\to\operatorname{Aut}(W)$ the associated Higgs vector bundle $\mathfrak W$  is semistable by Lemma \ref{adjsemi}  (since $G$ is central), and this is true after pullback to any curve.

 Conversely, given a $G$-torsor on $\cat{Higgs}_X^{\Delta,+}$, one builds a principal Higgs $G$-bundle $\fE$ as in \cite[Ch.~6]{Si1}. We prove that $\fE$ is
semistable. If $\mathfrak W$
is an associated Higgs vector  bundle  via a faithful representation, $\Ad(\fE)$ is a Higgs subbundle of
$\End(\mathfrak W)$. If $\mathfrak W$ is semistable,
since  $c_1(\Ad(\fE))=c_1(\End(\mathfrak W))=0$ the bundle   $\Ad(\fE)$ is semistable, so that $\fE$ is semistable as well. This is true after pullback to any curve, so that
$\fE$ is semistable  after pullback to any curve.
\end{proof}

\begin{remark}
In the case of principal (non-Higgs) $G$-bundles, let us denote
by $\mathbf{Vect}^{\Delta,+}_X$ the additive completion of the category of vector bundles that are semistable after pullback to any curve. Since in this case all three conditions in Theorem \ref{Miyahigher} are equivalent, $\mathbf{Vect}^{\Delta,+}_X$ is equivalent to the additive completion of the category of semistable vector bundles with vanishing discriminant.
\end{remark}

\bigskip\section{Appendix: Semistable and numerically\\ effective Higgs vector bundles}
\label{appendix}

Since our treatment of principal Higgs bundles relies quite heavily on previous work
on Higgs vector bundles, we provide here a short resume of the main results
in that theory. The main references are \cite{BG1,BG2,BH}, even though
the treatment we give here includes some modifications. We shall give here only a sketch of the main proofs, referring to \cite{BG1,BG2,BH} for a more detailed
and complete treatment.

\subsection{Ample and numerically effective Higgs bundles}

All varieties are projective varieties
over the complex field.
Let $V$ be a vector bundle of rank $r$ on $X$, and let $s$ be a positive
integer less than $r$. We shall denote by $\grass_s(V)$  the Grassmann bundle of $s$-planes in
$V$, with projection $p_s : \grass_s(V) \to X$.    There
is a universal exact sequence
\begin{equation}\label{univ}
0 \to S_{r-s,V} \xrightarrow{\psi} p_s^*(V) \xrightarrow{\eta} Q_{s,V}
\to 0
\end{equation}
of vector bundles on $\grass_s(V)$, with $S_{r-s,V}$ the universal rank $r-s$ subbundle and $Q_{s,V}$ the universal rank $s$ quotient bundle.

\begin{defin} A Higgs sheaf $\fV$ on $X$ is a coherent sheaf $V$ on $X$ endowed with
a morphism $\phi \colon V \to V \otimes \Omega_X$ of $\Oc_X$-modules such that
$\phi\wedge\phi=0$, where $\Omega_X$ is the cotangent sheaf to $X$. A Higgs subsheaf
$W$ of a Higgs sheaf $\fV=(V,\phi)$ is a subsheaf of $V$ such that $\phi(W)\subset
W\otimes\Omega_X$. A  Higgs bundle is a Higgs sheaf $\fV $ such that $V$ is a
locally-free $\Oc_X$-module. A Higgs sheaf $\fV=(V,\phi) $ is semistable (resp.~stable) if $V$ is
torsion-free, and $\mu(W)\le \mu(V)$ (resp. $\mu(W)< \mu(V)$) for every proper
nontrivial Higgs subsheaf $W$ of $\fV$.\end{defin}

Given a Higgs bundle $\fV $, we may construct closed subschemes
$\hgrass_s(\fV)\subset \grass_s(V)$ pa\-ram\-e\-ter\-iz\-ing rank
$s$ locally-free Higgs quotients, i.e., locally-free quotients of
$V$ whose corresponding kernels are $\phi$-invariant. We define
$\hgrass_s(\fV)$ \emph{(the Grassmannian of locally free rank $s$ Higgs quotients of   $\fV$)} as the closed subscheme of $\grass_s(V)$ where
the composed morphism
\begin{equation}\label{lambda}
(\eta\otimes1)\circ p_s^\ast(\phi) \circ \psi\colon S_{r-s,V}\to Q_{s,V}\otimes
 p_s^\ast\Omega_X
\end{equation}
vanishes. We denote by $\rho_s$ the projections $\hgrass_s(\fV)\to X$. The
restriction of \eqref{univ} to the scheme $\hgrass_s(\fV)$ provides the universal
exact sequence
$
0\to S_{r-s,\fV}\xrightarrow{\psi} \rho_s^\ast (V)\xrightarrow{\eta}
 Q_{s,\fV}\to 0,
$
and $Q_{s,\fV}$ is a rank $s$ universal Higgs  quotient vector bundle, i.e., for every
morphism $f\colon Y\to X$ and every  rank $s$ Higgs quotient $W$ of $f^\ast V$ there
is a morphism $\psi_W\colon Y\to \hgrass_s(\fV)$ such that $f=\rho_s\circ \psi_W$
and $W\simeq\psi_W^\ast (Q_{s,\fV})$.  

\begin{defin}  \label{moddef}  
A Higgs bundle $\fV$ of rank one is said to be Higgs-numerically
effective  (H-nef) if it is numerically effective in the usual sense. If
$\rk \fV \geq 2$ we require that:
\begin{enumerate} \item all bundles $Q_{s,\fV}$ are Higgs-nef; \item the line bundle $\det(V)$ is nef.
\end{enumerate}
If both $\fV$ and $\fV^\ast$ are Higgs-numerically effective, $\fV$ is said to be Higgs-numerically flat (H-nflat).
\end{defin}

Note that if
$\fV=(V,\phi)$, with $V$ nef  in the usual sense, than $\fV$ is H-nef. Moreover, if
$\phi=0$, the Higgs bundle $\fV=(V,0)$ is  H-nef if
and only if $V$ is nef in the usual sense.
 
\begin{prop} \label{properties}
Let $X$ be a smooth projective variety.
\begin{enumerate}
       \item If $f\colon Y \to X$ is a finite surjective morphism of smooth projective
    varieties, and $\fV$ is a Higgs bundle on $X$, then $\fV$ is H-ample (resp.~H-nef) if and only if $f^\ast\fV$ is H-ample (resp.~H-nef).
     \item Every quotient Higgs bundle of a H-nef Higgs bundle
    $\fV$ on $X$ is H-nef.
\end{enumerate}
\end{prop}

 The recursive condition in the definition of H-nefness may be actually expressed in terms of a simpler set of nefness conditions. Let us denote by $Q(s_1,\dots,s_k)_\fV$ the universal Higgs bundle obtained by
 taking the successive universal Higgs quotients of $\fV$, first of rank $s_k$, then
 $s_{k-1}$, all the way to rank $s_1$. The indexes $s_1,\dots,s_k$ satisfy
\begin{equation}\label{s} 1 \le s_1 < \dots < s_k < r.\end{equation}
 So  for instance, $Q_{1,Q_s,\fV}=Q(1,s)_\fV$.
 Moreover $Q(s_1,...,s_k)_\fV$ is a rank $s_1$ Higgs  bundle on $\hgrass_{s_1}(Q(s_2,\dots,s_k)_\fV)$.
The H-nefness condition for $\fV$ amounts to saying that the determinant bundles
$\det(Q(s_1,...,s_k)_\fV)$ are nef for all strings $s_1,\dots,s_k$ satisfying \eqref{s},
and that the line bundles  $Q(1,s_2\dots,s_k)_\fV$ are nef,  for all strings of the type
$(1,s_2,\dots,s_k)$.
 
 \begin{prop} Let $\fV$ be a  Higgs bundle  on $X$. The following conditions are equivalent.
 \begin{enumerate} \item $\fV$  is H-nef
 \item  for every $s$ satisfying $0<s<r=\rk(\fV)$
  the line bundle $Q_{1,Q_s,\fV}$ on $\hgrass_1(Q_{s,\fV})$ is nef,
  and for every string  of  integers $s_1,\dots,s_k$ such that $1 \le s_1 < \dots < s_k < r$,
  the line bundles $\det(Q(s_1\dots,s_k)_\fV)$ are nef.
  \end{enumerate}
\label{crit}\end{prop}
 
 \begin{proof} One has a (surjective)
morphism $$\rho_{s_2\dots,s_k}\colon \hgrass_1(Q(s_2\dots,s_k)_\fV) \to \hgrass_{s_k}(\fV),$$
and $Q(1,s_2\dots,s_k)_\fV$ turns out to be a rank one Higgs quotient
of $\rho_{s_2\dots,s_k}^\ast Q_{s_k,\fV}$. By universality, there is morphism
$$f_{s_2\dots,s_k}\colon \hgrass_{1}(Q(s_2,\dots,s_k)) \to \hgrass_1(Q_{s_k,\fV})$$
such that
\begin{equation}\label{pb}  Q(1,...,s_k)_\fV  \simeq f_{s_2\dots,s_k}^\ast Q(1,s_k)_\fV.\end{equation}

If $\fV$ is H-nef, the nefness of the determinant bundles $\det(Q(s_1\dots,s_k)_\fV)$
and of the line bundle $Q(1,s_k)_\fV$
holds by definition.  

Conversely, if the conditions in (ii) hold, the line bundles $Q(1,...,s_k)_\fV$
are nef as a consequence of \eqref{pb}, so that $\fV$ is H-nef.
\end{proof}

\subsection{Generalizing Miyaoka's semistability criterion}
In \cite{Mi} Miyaoka introduced a numerical class $\lambda$ in the projectivization
$\PP V  $ which, when $X$ is a curve, is nef if and only if $V$ is semistable.  
In the case of a  Higgs bundle $\fV$ on a smooth projective variety $X$, we introduce
the following generalizations of the class $\lambda$. These are numerical classes
in the Higgs Grassmannians $\hgrass_s(\fV)$:
$$\theta_{s, \fV}=[c_1(Q_{s,\fV})]-\frac sr\, \rho_s^\ast (c_1(V)) \in
N^1(\hgrass_s(\fV)),$$ where
 $\rho_s \colon \hgrass_s(\fV)\to X$ is  the
natural epimorphism.

Let $\Delta(V)$ be the
characteristic class
$$\Delta(V) =  c_2(V)-\frac{r-1}{2r}c_1(V)^2 = \frac{1}{2r}c_2(V\otimes V^\ast)\, .$$

\begin{thm} \label{1.3BHR}
Let $\fV$ be a Higgs bundle on a smooth projective variety. Consider the following conditions. \begin{enumerate}
    \item All classes $\theta_{s, \fV}$ are nef, for $0<s<r$.
    \item For any smooth projective curve $C$ in $X$, the restriction $\fV_{\vert C}$ is semistable.
      \item $\fV$ is semistable and $\Delta(V)=0$.
\end{enumerate}
Then conditions (i) and (ii) are equivalent, and they are both implied by condition (iii). 

\end{thm}

\begin{lemma} \label{curves}
A Higgs bundle   $\fV$  on  a smooth projective curve $C$ is semistable  if and only if  all classes $\theta_{s, \fV}$ are nef.
 \end{lemma}

\begin{proof} Assume $\fV$ is semistable.  If for some $s$ the  
class $\theta_{s, \fV}$ is not nef 
there is an irreducible curve $C'\subset \hgrass_{s}(\fV)$ which surjects onto  
$C$ and is such that  $C'\cdot\theta_{s,\fV}<0$.  As in the proof
of Theorem \ref{Miyahigher}, we may assume that $C'\to C$ is an isomorphism.
Denote   by $\fQ$ 
the restriction of $Q_{s,\fV} $ to $C'$, and let $\fV'=(p_s^\ast\fV)_{\vert C'}$,
where $p_s\colon\hgrass_s(\fV)\to C$ is the projection. 
$\fV'$ is semistable, and  we have 
$$ 0 > [C'] \cdot\theta_{s,\fV} = [C']\cdot (c_1(\fQ)-\frac{s}{r}\, {p_s}^{\ast}c_1(V))
= s (\mu(\fQ) - \mu(\fV'))
$$
but this contradicts the semistability of $\fV'$.

 If all classes $\theta_{s, \fV}$  are nef, let $\fV'$ be a rank $s$ Higgs quotient of $\fV$,
and let $\sigma\colon C\to\hgrass_s(\fV)$ be the corresponding section. Then
$$ 0 \le \theta_{s, \fV} \cdot [\sigma(C)] = s(\mu(V')-\mu(V))$$
so that $\fV$ is semistable.
\end{proof}

This implies that  conditions (i) and (ii) in Theorem \ref{1.3BHR} are equivalent.
 
\begin{lemma} Let $\fV$ be a Higgs bundle on a smooth
projective variety $X$. If the restriction of $\fV$ to any smooth curve $C$ is $X$
is semistable, and $c_1(V)=0$, then $\fV$ is H-nflat.
\label{curverestrHnef}
\end{lemma}
\begin{proof} We may assume that $X$ is a curve. Let $$\lambda_{s,\fV} = c_1(\cO_{\PP Q_{s,\fV}}(1))_{\vert \hgrass_1(Q_{s,\fV})}\,.$$
We show that for every $s$, with $0<s<r=\rk(\fV)$, the class $\lambda_{s,\fV}$ is nef.
Let $C$ be a curve in $ \hgrass_1(Q_{s,\fV})$. Possibily after a base change, 
we may assume that $C$ projects isomorphically onto a curve $C'$ in $\hgrass_s(\fV)$
and that this projects isomorphically onto $X$. We have the diagram
$$\xymatrix{
C \ar[r]^{j\ \ \ \  }& \hgrass_1(Q_{s,\fV}) \ar[d] \\
C'\ar[u]^\sigma \ar[r] & \hgrass_s(\fV)\ar[d] \\
& X \ar[ul]^{\sigma'} }$$
If we let $\mathcal L = (j\circ\sigma\circ\sigma')^\ast \cO_{\PP Q_{s,\fV}}(1)$
then $\mathcal L$ is a rank one Higgs quotient of $\fV$, so that
$\mu(\fV)\le \deg(\mathcal L)$. But since $\deg (\mathcal L) - \mu(\fV) = [C]\cdot \lambda_{s,\fV}$,
we have that $\lambda_{s,\fV}$ is nef.

Now we prove that $\fV$ is H-nflat. In this actually enough to prove that $\fV$ is H-nef, and then
apply the same reasoning to $\fV^\ast$. Now, $\lambda_{s,\fV}$ is the first Chern class
of the line bundle $Q_{1,Q_{s,\fV}}$. In view of Proposition \ref{crit},
it remains only to show that the determinant bundles $\det(Q(s_1,\dots,s_k)_\fV)$
are nef for all strings of integers $s_1,\dots,s_k$ as in Proposition \ref{crit}.
We note that $Q(s_1,\dots,s_k)_\fV$ is a Higgs bundle on $\hgrass_{s_1}(Q(s_2,\dots,s_k)_\fV)$
and that there is a morphism $\rho\colon\hgrass_{s_1}(Q(s_2,\dots,s_k)_\fV)\to X$ such
that $Q(s_1,\dots,s_k)_\fV$ is a Higgs quotient of $\rho^\ast\fV$. Therefore, by universality,
there is morphism $g_{s_1,s_2\dots,s_k}\colon \hgrass_{s_1}(Q(s_2,\dots,s_k)_\fV)\to \hgrass_{s_1}(\fV)$
such that $Q(s_1,\dots,s_k)_\fV\simeq {g_{s_1,s_2\dots,s_k}}^\ast Q_{s_1,\fV}$. We have now
$$c_1(Q(s_1,\dots,s_k)_\fV) = {g_{s_1,s_2\dots,s_k}}^\ast(c_1(Q_{s_1,\fV}))= {g_{s_1,s_2\dots,s_k}}^\ast \theta_{s_1,\fV}$$
since $c_1(V)=0$, so that $\det(Q(s_1,\dots,s_k)_\fV)$ is nef since $\theta_{s_1,\fV}$ is nef  by Lemma \ref{curves}.
\end{proof}

\begin{prop}  An H-nflat    Higgs bundle $\fV$ on a smooth
projective variety $X$ is semi\-stable. \label{flatisss}
\end{prop} 
\begin{proof} As the restriction of an H-nflat Higgs bundle $\fV$
to a closed subvariety of $X$ is H-nflat,  we may assume that $X$ is a curve. Since $\fV$ is in particular H-nef,
all universal Higgs quotients $Q_{s,\fV}$ are H-nef, and then the determinant
bundles $\det Q_{s,\fV}$ are nef. On the other hand, since $\det(V)$
is numerically flat, we have $c_1(V)=0$. Therefore the classes $\theta_{s,\fV}$
are nef. We conclude by Lemma \ref{curves}.
\end{proof}

  To conclude the proof of Theorem  \ref{1.3BHR}  one needs to prove that
  condition (iii) implies condition (ii).  This is proved in \cite{BH}. 
  
  \begin{remark} For (non-Higgs) vector bundles, the three conditions
  in Theorem \ref{1.3BHR} are equivalent \cite{BH}.\end{remark}

\end{document}